\documentclass[11pt,reqno]{amsart}
\usepackage{mathpazo} 
\normalfont
\usepackage[T1]{fontenc}

\usepackage[utf8]{inputenc}
\usepackage{amsmath}
\usepackage{amsfonts}
\usepackage{dirtytalk}
\usepackage{amssymb}
\usepackage{amsthm}
\usepackage{mathtools}
\usepackage{stmaryrd}
\usepackage{tikz}
\usepackage{tikz-cd}
\usepackage[all]{xy}
\usepackage{enumitem}
\usepackage{xcolor}
\usepackage[margin=1.0in]{geometry} 

\newtheorem{thm}{Theorem}[section]
\newtheorem{theorem}[thm]{Theorem}
\newtheorem{lemma}[thm]{Lemma}

\newtheorem{corollary}[thm]{Corollary}
\newtheorem{prop}[thm]{Proposition}

\theoremstyle{definition}

\newtheorem{definition}[thm]{Definition}
\newtheorem{remark}[thm]{Remark}
\newtheorem{example}[thm]{Example}
\newtheorem{notation}[thm]{Notation}

\newtheorem{innercustomthm}{Theorem}
\newenvironment{customthm}[1]
{\renewcommand\theinnercustomthm{#1}\innercustomthm}
{\endinnercustomthm}

\def\cocoa{{\hbox{\rm C\kern-.13em o\kern-.07em C\kern-.13em o\kern-.15em A}}}

\def\B{{\mathbb{B}}}

\def\N{{\mathbb{N}}}

\newcommand{\mc}[1]{\mathcal{#1}}

\newcommand{\mf}[1]{\mathfrak{#1}}

\def\supp{\textrm{Supp}}

\def\<{\langle}
\def\>{\rangle}

\begin{document}

\title{Prime ideals in the Boolean polynomial semiring}

\author{Kalina Mincheva}
\address{Department of Mathematics, Tulane University, New Orleans, LA 70115}
\email{kmincheva@tulane.edu}

\author{Naufil Sakran}
\address{Department of Mathematics, Tulane University, New Orleans, LA 70115}
\email{nsakran@tulane.edu}



\begin{abstract}
In this article, we disprove a conjecture of F. Alarcón and D. Anderson and give a complete classification of the prime ideals in the one variable polynomial semiring with coefficients in Boolean semifield. 
We group the prime ideals of $\B[x]$ into three classes, indexed by integers.
\end{abstract}

\maketitle

\section{Introduction}

The study of prime ideals in polynomial rings 
plays a fundamental role in algebra and algebraic geometry, providing insights into their structure, decomposition properties, and applications in various mathematical disciplines. In classical commutative algebra, the classification of prime ideals in polynomial rings over fields has been extensively studied, leading to deep connections with algebraic geometry and number theory. Moreover, there is an algorithm to verify if an ideal of such a ring is prime. However, if we replace the coefficient field with an additively idempotent semifield, the problem of classifying primes ideals becomes very different and challenging due to the lack of additive inverse in additively idempotent semirings.


One of the main motivations for the study of prime ideals of polynomial semirings with coefficients in idempotent semifields comes from tropical geometry. The goal is to understand the underlying algebra of tropical spaces in order to strengthen existing tropical methods. We note that the k-ideals or subtractive ideals {\cite{Gol92}} that are usually studied in the context of semirings are a very small and restrictive class of ideals in polynomial (additively idempotent) semirings and do not contain interesting geometric information. Instead, of greater interest have been other classes of ideals such as tropical ideals or closed ideals, defined and studied in \cite{FGGM24}, \cite{JM25}, \cite{MR18}. In \cite{JM25}, the authors study prime closed and prime tropical ideals using congruence methods.

Of particular interest is the polynomial semiring with coefficients in the Boolean semifield $\B$, which is the smallest idempotent semifield. The Boolean semifield has two elements $\{1,0\}$, where $1$ is the multiplicative identity, $0$ is the additive identity and $1+1 = 1$. This article focuses on classifying the prime ideals of the semiring $\B[x]$. The prime ideals of $\B[x]$ behave very differently than those of $k[x]$ where $k$ is any field. For example, the ideal $\<1+x\>$ is not prime in $\B[x]$.

One of the first studies of prime ideals in $\B[x]$ was done by F. Alarcón and D. Anderson in \cite{AA94}, where they gave the following partial classification of the prime ideals in $\B[x]$. They relate the ideals of $\B[x]$ to prime subsets of $\N$. To a subset $\mf{A}$ of $\N$ we can associate two ideals of $\B[x]$
\begin{align*}
    \mc{I}_{\mf{A}}&=\< x,\{1+x^a\mid a\in \mf{A}\}\> \text{ and } \\
    \mc{J}_{\mf{A}}&=\< x,\{1+x^a\mid a\in \mf{A}\}\> - \{x^mf\mid m\geq 0,\,f\notin \mc{I}_{\mf{A}}\}.
\end{align*}

\begin{definition}
    We call a subset $\mf{A}$ of $\N$ prime if $a+b\in \mf{A}$ then either $a\in \mf{A}$ or $b\in \mf{A}$.
\end{definition}

\begin{customthm}{\cite[Theorem 11, for ideals containing $\<x\>$]{AA94}} 
The ideal $\mc{I}_{\mf{A}}$ is prime if and only if $\mf{A}$ is a prime subset. If $\mc{P}$ is a prime ideal in $\B[x]$ containing the ideal $\<x\>$ such that $\mf{A}_{\mc{P}}=\{a\mid 1+x^a\in \mc{P}\}$, then $\mc{P}=\mc{I}_{\mf{A}_{\mc{P}}}$.
\end{customthm}

The above theorem classifies all the prime ideals of $\B[x]$ that contain the ideal $\<x\>$. For the classification for prime ideals not containing $\<x\>$, Alarcón and Anderson showed the following.

\begin{customthm}{\cite[Theorem 11, for ideals \textbf{not} containing $\<x\>$]{AA94}} 
The ideal $\mc{J}_{\mf{A}}$ is prime if and only if $\mf{A}$ is a prime subset.
\end{customthm}

In \cite{AA94}, the authors conjecture that if $\mc{P}$ is a prime ideal in $\B[x]$, not containing the ideal $\<x\>$ then $\mc{P}=\mc{J}_{\mf{A}}$ for some prime subset $\mf{A}$ in $\N$. In Example~\ref{counterexample_AA94}, we will show that this is not the case, and the following ideal
\[\mc{J} = \left\langle 1+x,\{1+x^a+x^{a+1}:a\geq 2\}    \right\rangle\]
does not correspond to any ideal of the form $\mc{J}_{\mf{A}}$ for any prime subset $\mf{A}$. 

In this article, we will give a complete classification of prime ideals of $\B[x]$ that do not contain the principal ideal $\<x\>$. Our first main result indicates that prime ideals of $\B[x]$ can be parametrized by prime subsets of $\N$.

\begin{customthm}{A}\label{A}
Let $\mc{P}$ be a prime ideal in $\B[x]$. Then the set 
$\mathfrak{A}_{\mc{P}}:=\{a:1+x^a\in \mc{P}\}$
is a prime subset.
\end{customthm}

\noindent Theorem \ref{A} suggests that we have a surjective map from the set of prime ideals of $\B[x]$ to the set of prime subsets of $\N$, given by $\mc{P}\longmapsto \mf{A}_{\mc{P}}$. This motivates us to partition the set of prime ideals of $\B[x]$ by their corresponding prime subsets, and say that a prime ideal $\mc{P}$ belongs to \say{class} $[\mf{A}]$ if $\mf{A}_{\mc{P}}=\mf{A}$. Therefore, for a specific prime subset $\mf{A}$ of $\N$, we aim to determine all the prime ideals in the class $[\mf{A}]$. 

Following this idea, we study how the structure of a prime subset $\mf{A}$ affects the type of prime ideals in $[\mf{A}]$. In Lemma~\ref{A prime -> N-A submonoid}, we show that $\mf{A}$ is a prime subset if and only if $\N-\mf{A}$ is a semigroup in $\N$. Now, since semigroup of $\N$ are finitely generated, there exists a minimal generating set $\{b_1,b_2,\cdots,b_s\}$ of the semigroup $\N-\mf{A}$. Denote by $d$ the greatest common divisor of $b_1,\cdots,b_s$, i.e., $d=\gcd(b_1,\cdots,b_s)\geq 1$. One sees that prime subsets in $\N$ can be parametrized by the set of positive integers as follows: for any fixed $d\geq 1$, let $[d]$ denote the set of all the prime subsets $\mf{A}$, for which $d=\gcd(b_1,\cdots,b_s)$, where $\N-\mf{A}=\<b_1,\cdots,b_s\>$. Hence, viewing prime subsets $\mf{A}$ in terms of their corresponding semigroup $\N-\mf{A}$, we are able partition the set of all prime subsets indexed by the set of positive integers $\N$. For example, for $d=1$, the class $[d]$ consists of all finite prime subsets $\mf{A}$ (i.e $|\mf{A}|<\infty$).

Using this terminology, we provide a complete classification of prime ideals in terms of the class $[d]$ to which they belong. By combining our results with those of F. Alarcón and D. Anderson in \cite{AA94}, we state below the main theorem of this article, which completes the classification. By convention, the empty set is also a prime subset. 

\begin{notation}
    Let $\mf{A}$ be a non-empty prime subset and let $f(x)=1+x^{m_1d}+\cdots +x^{m_rd}\in\B[x]$. We say $f$ satisfies property $\star$ if the set $\{m_1d,\ldots,m_rd\}$ satisfies the following:
    \begin{itemize}
        \item $m_id\in\mf{A}$ for some $i$ but $\{m_1d,\ldots,m_rd\}\not\subset\mf{A}$,
        \item and there exists $t$ such that $|m_t-m_j|d\notin\mf{A}$ for all $j$.  
    \end{itemize}
\end{notation}

For the remainder of the section, we say $f(x)=1+x^{m_1d}+\cdots+x^{m_rd}$ satisfy $\star$ if $m_id\in\mf{A}$ for some $i$ but $\{m_1d,\ldots,m_rd\}\not\subset \mf{A}$ and there exists $t$ such that $|m_t-m_j|d\notin\mf{A}$ for all $j=1,\ldots,r$.  

\begin{customthm}{B}\label{D}
    Let $\mf{A}$ be a non-empty prime subset in $\N$ and let $\mc{S}$ be the semigroup $\N-\mf{A}$ generated by $\<a_1,\ldots,a_r\>$. Let $d=\gcd(a_1,\ldots,a_r)$. 
    Let $\mc{J}$ be a prime ideal not containing $\<x\>$ in $\B[x]$. 
    \begin{enumerate}
        \item if $d\notin\mf{A}$, then $\mc{J}=\mc{J}_{\mf{A}}$.
        
        \item if $d\in \mf{A}$ then either $\mc{J}=\mc{J}_{\mf{A}}$ or $\mc{J}$ is of the form
        \[\mc{J}= \langle\{1+x^{m}+x^{m+a}:a\in\mf{A}, m\geq 0\},\mc{Q}\rangle\]
        for some $\mc{Q}\subseteq\{f\mid f\textrm{ is irreducible and satisfy $\star$}\}$.
        \end{enumerate}
\end{customthm}

The classification of prime ideals in the multivariable Boolean semiring $\B[x_1,\ldots,x_n]$ remains an open problem. In the final section, we present some partial results on the case $\B[x_1,\ldots,x_n]$. 

\addtocontents{toc}{\protect\setcounter{tocdepth}{2}}

\section*{Acknowledgments}
K.M. acknowledges the support of the Simons Foundation, Travel Support for Mathematicians MPS-TSM-00008148. N.S. acknowledges partial support from the Louisiana Board of Regents grant, contract no. LEQSF(2023-25)-RD-A-21.

\section{Notations and Preliminaries}\label{not_and_prem}

Let $\N=\{1,2,3,\rightarrow\}$ where the arrow $\rightarrow$ indicates all the subsequent integers. 

\begin{definition}\label{prime subset}
    A subset $\mf{A}$ in $\N$ is said to be \textit{prime} if for any $a+b\in \mf{A}$, we have either $a\in \mf{A}$ or $b\in \mf{A}$.    
\end{definition}
\begin{remark}
    The empy set is a prime subset by default. Also, for any non-empty prime subset, we must have $1\in\mf{A}$. 
\end{remark}
\begin{lemma}\label{A prime -> N-A submonoid}
    A subset $\mf{A}$ in $\N$ is prime if and only if $\N-\mf{A}$ is a semigroup in $\N$.    
\end{lemma}
\begin{proof}
    Let $\mf{A}$ be a prime subset and let $x,y\in\N-\mf{A}$. Then if $x+y\in\mf{A}$, it would imply $x\in\mf{A}$ or $y\in\mf{A}$, which is a contradiction. So, we must have $x+y\in\N-\mf{A}$. Therefore, $\N-\mf{A}$ is a semigroup.

    Conversely, suppose $\N-\mf{A}$ is a semigroup and let $x+y\in\mf{A}$. If $x\notin\mf{A}$, then we must have $y\in\mf{A}$ as if it were not the case, then we would have $x+y\in\N-\mf{A}$, a contradiction. 
\end{proof}

We say that a non-empty set $R$ is a {\it semiring} with respect to the binary operations $(+,\cdot)$ if it satisfies the same axioms as those in a ring, except we do not require the addition to be invertible. A {\it semifield} is a commutative semiring with multiplicative unit in which all nonzero elements have multiplicative inverse.

We will denote by $\B$ the \textit{Boolean semifield}. Its underlying set is $\{0,1\}$ with addition and multiplication given by the rules: 
\[0+0=0,\qquad0+1=1+0=1+1=1,\qquad 0\cdot 0=0,\qquad 0\cdot 1=1\cdot 0=0,\qquad 1\cdot 1=1.\]

By $\B[x]$ we denote the polynomial semiring with coefficients in $\B$. Let $\mc{I}$ be a subset of $\B[x]$. We say $\mc{I}$ is an \textit{ideal} in $\B[x]$ if for any $f,g\in \mc{I}$, we have $f+g\in\mc{I}$ and $\B[x]\cdot\mc{I}\subseteq\mc{I}$. Moreover, $\mc{I}$ is said to be a \textit{prime ideal} if for any $fg\in\mc{I}$, we have either $f\in \mc{I}$ or $g\in\mc{I}$. We list here some properties of $\B[x]$ important to our discussion.

\begin{itemize}
    \item $\B[x]$ is not UFD as $(1+x+x^3)(1+x^2+x^3)=(1+x)(1+x^2)(1+x^3)$.
    \item The ideal $\< 1+x\>$ is not prime as $(1+x^2+x^3)(1+x+x^3)=(1+x)^6$ and neither of the factors on the left-hand side are in the ideal.
    \item If $f(x)\in\B[x]$ is a degree $N$ polynomial with $1\in\supp(f)$, then 
    \[f(x)(1+x)^N=(1+x)^{2N}.\]
    This implies that every prime ideal of $\B[x]$ contains the polynomial $1+x$, except for the prime ideal $\<x\>$.
\end{itemize}

We can divide the primes of $\B[x]$ into two categories:
\begin{enumerate}
    \item[(I)] those that contain $\<x\>$,
    \item[(II)] and those that do not contain $\<x\>$.
\end{enumerate}
Throughout the paper, if a prime ideal lies in category (I), we denote it by $\mc{I}$ ($\mc{I}',\mc{I}_{\mf{A}}$, etc.) and if the prime ideal lies in category (II), we denote it by $\mc{J}$ ($\mc{J}',\mc{J}_{\mf{A}}$, etc.). Observe that we can express the ideal $\mc{J}_{\mf{A}}:=\mc{I}_{\mf{A}} - \{x^mf\mid m\geq 0,\,f\notin \mc{I}_{\mf{A}}\}$ also as 
\begin{align*}
    \mc{J}_{\mf{A}} &= \{x^mf\mid m\geq0,\,f\in\mc{I}_{\mf{A}}\textrm{ with $1\in\supp(f)$}\}\cup\{0\}
    \\
    &= \<\{1+x^a+x^m\mid a\in\mf{A}\textrm{ and }m\geq 0\}.\>
\end{align*}
We will mostly use the above description of $\mc{J}_{\mf{A}}$ throughout the article. The following example contradicts the conjecture of F. Alarcón and D. Anderson in \cite{AA94}.
\begin{example}\label{counterexample_AA94}
    The ideal 
    \[\mc{J} = \left\langle \{1+x^a+x^{a+1}:a\geq 0\}    \right\rangle\]
    is a prime ideal in category (II) and it is not of the form $\mc{J}_{\mf{A}}$ for any prime subset $\mf{A}$.
    
    We verify that $\mc{J}$ is a prime ideal. Observe that an alternative way of writing $\mc{J}$ is:
    \[\mc{J}=\left\{f\in \B[x]: \textrm{ if $x^a$ is the leading monomial of $f$ then $x^{a-1}+x^a\in \supp(f)$}\right\}.\]
    Let $f=\sum_ix^{a_i}$ and $g=\sum_jx^{b_j}$ be elements of $\B[x]$ with leading monomials $x^a$ and $x^b$, respectively. Let $fg\in \mc{J}$. By the definition above, $x^{a+b-1}+x^{a+b}\in \supp(fg)$. But then this implies $x^{a-1}\in \supp(f)$ or $x^{b-1}\in \supp(g)$, implying that $f\in \mc{J}$ or $g\in\mc{J}$. Hence, $\mc{J}$ is a prime ideal. Here $\mf{A}_{\mc{J}}=\{a\mid 1+x^a\in \mc{J}\}=\{1\}$.
    
    Finally, $\mc{J}$ is not of the form 
    \begin{align*}
        \mc{J}_{\mf{A}_{\mc{J}}}= \langle\{1+x+x^m:m\geq 0\} \rangle,
    \end{align*}
    since $1+x^2+x^3\in \mc{J}$ but $1+x^2+x^3\notin\mc{J}_{\mf{A}_{\mc{J}}}$. 
\end{example}

\addtocontents{toc}{\protect\setcounter{tocdepth}{2}}

\section{Prime ideals and prime subsets}\label{prime ideals and prime subsets}

In this section, we prove that all the prime ideals in $\B[x]$ can be parametrized by prime subsets of $\N$.

\begin{theorem}\label{prime ideal -> prime subset}
    Let $\mc{P}$ be a prime ideal in $\B[x]$. Then the set 
    \[\mf{A}:=\{a\mid 1+x^a\in \mc{P}\}\]
    is a prime subset.
\end{theorem}
\begin{proof}
    If $\mf{A}=\emptyset$ or $\mf{A}=\{1\}$, then we have nothing to prove. Now let $|\mf{A}|>1$ and suppose $\mf{A}$ is not a prime subset, then there exists $a\in \mf{A}$ and $r,s\in\N\backslash \mf{A}$ such that $r+s=a$. Assume that $r\leq s$. We claim that the following 
    \begin{align}\label{eq: product}
        (1+x^r)^s(1+x^s)^s(1+x^a)=(1+x^r)^{s+1}(1+x^s)^{s+1}.
    \end{align}
    holds in $\B[x]$. It suffices to show that all possible linear combinations $ir+i's$, where $0\leq i,i'\leq s+1$ can be written as the linear combination of $jr+j's+j''a$ for $0\leq j,j'\leq s$ and $0\leq j''\leq 1 $ and vice versa. To see this, consider a linear combination $ir+i's$, where $0\leq i,i'\leq s+1$.
    \begin{enumerate}
        \item If $i,i'\leq s$ then we can take $j=i$ and $j'=i'$.
        \item If $i=s+1$ and $i'=0$ then we can take $j=1$ and $j'=r$.
        \item If $i'=s+1$ and $i=0$ then since $i's=rs+(i'-r)s,$  we can take $j=s$ and $j'=i'-r$.
        \item If $i=s+1$ and $0<i'\leq s$ then since $ir+i's=sr+(i'-1)s+a,$ we can take $j=s,j'=i'-1$ and $j''=1$. 
        \item If $0<i \leq s$ and $i'=s+1$ then since $ir+i's=(i-1)r+s^2+a$ we can take $j=i-1,j'=s$ and $j''=1$. 
        \item If $i = s+1$ and $i'= s+1$ then since $ir+i's=(i-1)r+(i'-1)s+a$ we can take $j=i-1,j'=i'-1$ and $j''=1$. 
    \end{enumerate}
    Similarly, consider an arbitrary linear combination $jr+j's+j''a$ for $0\leq j,j'\leq s$ and $0\leq j''\leq 1$.
    \begin{enumerate}
        \item If $j,j'\leq s$ and $j''=0$ then we can simply take $i=j$ and $i'=j'$.
        \item If $j,j'\leq s$ and $j''=1$ then we can simply take $i=j+1$ and $i'=j'+1$.
    \end{enumerate}
    
    This proves that Equation~(\ref{eq: product}) holds. The left-hand side of Equation~(\ref{eq: product}) is in $\mc{P}$ but neither $1+x^r$ nor $1+x^s$ is in $\mc{P}$ as $r,s\notin\mf{A}$. This contradicts our assumption that $\mf{A}$ is a prime set.
\end{proof}

\addtocontents{toc}{\protect\setcounter{tocdepth}{2}}

\section{Classification of the prime ideals of $\B[x]$}\label{Classification of prime ideals}

Let $\mf{A}\subseteq \N$ be a prime subset and let $\{da_1,\ldots,da_r\}$ generates the semigroup $\N\backslash\mf{A}$, where $\gcd(a_1,\ldots,a_r)=1$. Note that if $d=1$, then $\mf{A}$ is a finite set, and if $d>1$, then $\mf{A}$ is an infinite set. We say a prime subset $\mf{A}$ belongs to the class $d$ if $\N\backslash\mf{A}=\langle da_1,\ldots,da_r\rangle$ with $\gcd(a_1,\ldots,a_r)=1$. By Theorem~\ref{prime ideal -> prime subset}, a prime ideal $\mc{P}$ belongs to class $d$ if $\mf{A}_{\mc{P}}$ is in class $d$.

In this section, we prove that there are at most two types of prime ideals in category (II) with respect to each $d\in\N$. This will complete the classification of all prime ideals of $\B[x]$. Throughout the section, $\mc{I},\mc{I}_{\mf{A}}$ and $\mc{J},\mc{J}_{\mf{A}}$ will denote prime ideals in category (I) and category (II), respectively. For convenience, we write $\mf{A}$ for $\mf{A}_{\mc{J}}$.

\subsection{The case $d\in \mf{A}$}

Let $\mc{J}$ be a prime ideal with associated prime subset 
\[\mf{A}=\{a\in\N\mid 1+x^a\in\mc{J}\}.\]
Let $\N\backslash\mf{A}=\langle da_1\ldots,da_n\rangle$ with $\gcd(a_1,\ldots,a_n)=1$, $\mc{S}=\langle a_1,\ldots,a_r\rangle$, and $\alpha=\max\{\N\backslash\mc{S}\}+1$. Note that $(\alpha+i)d\in\N\backslash\mf{A}$ for all $i\geq 0$. Assume $d\in\mf{A}$. 

\begin{lemma}\label{m or m-a in A implies containtment}
    Let $a\in\mf{A}$. If $m\in\mf{A}$ or $|m-a|\in\mf{A}$, then 
    \[1+x^a+x^m\in\mc{J}\quad\textrm{and}\quad 1+x^{|m-a|}+x^m\in\mc{J}.\]
\end{lemma}
\begin{proof}
    Suppose $m\in\mf{A}$. Then clearly $(1+x^a)+(1+x^m)=1+x^a+x^m\in \mc{J}$. If $m-a>0$, we have $(1+x^{a})x^{m-a}+1+x^m=1+x^{m-a}+x^m\in\mc{J}$. Similarly, if $m-a<0$, we have $(1+x^{m})x^{a-m}+1+x^a=1+x^{a-m}+x^m\in\mc{J}$.

    Now, suppose $|m-a|\in\mf{A}$. Then clearly $1+x^{|m-a|}+x^m\in\mc{J}$. If $m-a>0$, then $(1+x^{m-a})x^a+1+x^a=1+x^a+x^m\in\mc{J}$. Similarly, if $m-a<0$, then $(1+x^{a-m})x^m+1+x^a=1+x^a+x^m\in\mc{J}$. 
\end{proof}

\begin{lemma}\label{d nmid a}
   If $a\in\mf{A}$ such that $d\nmid a$ then for every $m>a$, either $m-a\in\mf{A}$ or $m\in\mf{A}$. In particular,
    \[\langle\{1+x^a+x^m\mid m\geq 1\},\{1+x^{|m-a|}+x^{m}\mid m\geq 1\}\rangle\subseteq \mc{J}.\]
\end{lemma}
\begin{proof}
    If $|m-a|,m\notin \mf{A}$ then $d|a$, a contradiction. The rest follows from Lemma~\ref{m or m-a in A implies containtment}.
\end{proof}

\begin{lemma}\label{{m-a,m}_in_J}
    Let $a\in\mf{A}$. If $m,|m-a|\notin\mf{A}$ then either $1+x^a+x^m\in\mc{J}$ or $1+x^{|m-a|}+x^m\in\mc{J}$.
\end{lemma}
\begin{proof}
    Using the fact that
    \[(a+b+c)(ac+bc+ab)=(a+b)(b+c)(a+c),\]
    we have 
    \[(1+x^a+x^m)(1+x^{|m-a|}+x^m)=(1+x^a)(1+x^m)(1+x^{|m-a|})\in \mc{J}.\]
    Since the right-hand side of the above equation is in $\mc{J}$ and $\mc{J}$ is prime, we must have $1+x^a+x^m\in \mc{J}$ or $1+x^{|m-a|}+x^m\in \mc{J}$.
\end{proof}

\begin{lemma}\label{productRule1}
    If $m-a> \alpha$, then 
    \[(1+x^{\alpha d})^{m}\prod_{i=1}^{\alpha-1}(1+x^{(\alpha+i)d}) = (1+x^{(m-a)d}+x^{md}) + x^{\alpha d}(1+x^d)^{q-2\alpha} + (1+x^{ad}+x^{(m-a)d})x^{(q-m+a)d}\]
    for some positive integer $q$.
\end{lemma}
\begin{proof}
    Expanding the product, we have
    \begin{align*}
        (1+x^{\alpha d})^{m}\prod_{i=1}^{\alpha-1}(1+x^{(\alpha+i)d}) &= 1 + x^{\alpha d} + x^{(\alpha+1)d} + \cdots + x^{(q-\alpha) d}+x^{qd}
        \\
        &= 1 + \sum_{i=0}^{q-2\alpha} x^{(\alpha+i)d}+x^{qd}
        \\
        &=(1+x^{(m-a)d}+x^{md}) + x^{\alpha d}(1+x^d)^{q-2\alpha} + (1+x^{ad}+x^{(m-a)d})x^{(q-m+a)d}
    \end{align*}
    where $q=m\alpha+\sum_{i=1}^{\alpha-1}(\alpha+i)=\frac{\alpha(2m+3\alpha-3)}{2}$.
\end{proof}

\begin{lemma}\label{productRule2} 
    Let $m$ be an integer such that $\alpha\leq m<\alpha+a$. Then
    \begin{align*}
        (1+x^{(m-a)d})(1+x^{\alpha d})^m\prod_{i=1}^{\alpha-1}(1+x^{(\alpha+i)d}) &= (1+x^{(m-a)d}+x^{md})+x^{\alpha d}(1+x^d)^{q-2\alpha+m-a} \\
        & +(1+x^{ad}+x^{(q-\alpha)d})x^{\alpha d} 
        +(1+x^{ad}+x^{(q+m-a-\alpha)d})x^{\alpha d}
    \end{align*}
    for some positive integer $q$.
\end{lemma}
\begin{proof} 
    Expanding the product, we have
    \begin{align*}
        (1+x^{(m-a)d})(1+x^{\alpha d})^m\prod_{i=1}^{\alpha-1}(1+x^{(\alpha+i)d}) &= 1+x^{(m-a)d}+x^{\alpha d}+\cdots+x^{(q-\alpha+m-a)d}+x^{qd}+x^{(q+m-a)d}\\
        & =  1 + x^{(m-a)}+x^{md}+\sum_{i=0}^{q-2\alpha+m-a}x^{(\alpha+i)d}+x^{qd}+x^{(q+m-a)d}
        \\
        &= (1+x^{(m-a)d}+x^{md})+x^{\alpha d}(1+x^d)^{q-2\alpha+m-a} \\
        & +(1+x^{ad}+x^{(q-\alpha)d})x^{\alpha d} 
        +(1+x^{ad}+x^{(q+m-a-\alpha)d})x^{\alpha d}
    \end{align*}
    where $q=m\alpha+\sum_{i=1}^{\alpha-1}(\alpha+i)$.
\end{proof}

\begin{lemma}\label{productRule3}
    Let $m$ be an integer such that $a<m<\alpha$. Let $k$ be the smallest positive integer such that $km-a\geq \alpha$. Then 
    \begin{align*}
        (1+x^{(m-a)d})(1+x^{md})^{k-1}(1+x^{\alpha d})^{m} \prod_{i=1}^{\alpha-1}(1+x^{(\alpha+i)d}) &= (1+x^{(m-a)d}+x^{md}) + x^{\alpha d}(1+x^d)^{q-2\alpha+km-a} 
        \\
        &+(1+x^{ad})\left(\sum_{i=1}^{k-1}(x^{(im-a)d} + x^{(q+(im-a))d} )\right)
        \\
        &+(1+x^{ad}+x^{(a+\alpha)d})x^{(q-\alpha+km-2a)d}.
    \end{align*}
\end{lemma}
\begin{proof}
    Assume $m>a$. Let the $g$ denote the left side of the equation. Expanding the product, we have
    \begin{align*}
        g &= 1+ x^{(m-a)d}+x^{md}+x^{(2m-a)d} +x^{2md} +\cdots + x^{((k-1)m-a)d}
        \\
        & +x^{(k-1)md}+x^{\alpha d}+x^{(\alpha+1)d} + x^{(\alpha+2)d} + \cdots+x^{(km-a)d}+\cdots 
        \\
        & +x^{(q-\alpha+km-a)d}+x^{(q+m-a)d} + x^{(q+m)d}+x^{(q+2m-a)d}+x^{(q+2m)d}
        \\
        &+\cdots+ x^{(q+(k-1)m-a)d}+ x^{(q+(k-1)m)d} +x^{(q+km-a)d}
        \\
        &= 1+\sum_{i=0}^{k-1} (x^{imd} + x^{((i+1)m-a)d}) + \sum_{i=0}^{q-2\alpha+km-a} x^{(\alpha+i)d} + \sum_{i=0}^{k-1} (x^{(q+im)d} + x^{(q+(i+1)m-a)d})
        \\
        &= (1+x^{(m-a)d}+x^{md}) + x^{\alpha d}(1+x^d)^{q-2\alpha+km-a} 
        +(1+x^{ad})\left(\sum_{i=1}^{k-1}(x^{(im-a)d} + x^{(q+(im-a))d} )\right)
        \\
        &+(1+x^{ad}+x^{(a+\alpha)d})x^{(q-\alpha+km-2a)d},
    \end{align*}
    where $q=m\alpha+\sum_{i=1}^{\alpha-1}(\alpha+i)$. Note that $x^{qd}$ lies in the sum $\sum_{i=0}^{q-2\alpha+km-a} x^{(\alpha+i)d}$, since $km-a\geq\alpha$. 
\end{proof}

\begin{prop}\label{1+x^d+x^{md} contained}
    If there exists a positive integer $m$ for which $1+x^d+x^{md}\in \mc{J}$ but $(m-1)d,md\notin \mf{A}$ then the set $\{1+x^d+x^s \mid s\geq 1\}$ is contained in $\mc{J}$.
\end{prop}
\begin{proof}
    By Lemma~\ref{d nmid a}, we have  $\{1+x^d+x^s\mid d\nmid s\}\subseteq \mc{J}$. So, we only need to show the inclusion for $d\mid s$. Suppose $m$ is the smallest positive integer such that $(m-1)d,md\notin\mf{A}$ and $1+x^d+x^{md}\in \mc{J}$ and let $n$ be the smallest positive integer for which $(n-1)d,nd\notin\mf{A}$ with $1+x^d+x^{nd}\notin \mc{J}$. By Lemma~\ref{{m-a,m}_in_J}, we know $1+x^{(n-1)d}+x^{nd}\in \mc{J}$. Suppose $n>m$ and consider the following cases:
    \vspace{0.1in}
    
    \noindent\textbf{Case 1:} Suppose $n-1>\alpha$. By Lemma~\ref{productRule1}, 
    \[(1+x^{\alpha d})^n\prod_{i=1}^{\alpha-1}(1+x^{(\alpha+i)d}) = (1+x^{(n-1)d}+x^{nd}) + x^{\alpha d}(1+x^d)^{q-2\alpha} + (1+x^d+x^{(n-1)d})x^{(q-n+1)d}\]
    lies in $\mc{J}$ as $1+x^d+x^{(n-1)d}\in\mc{J}$ by hypothesis. However, since $\mc{J}$ is a prime ideal, this implies $(\alpha+i)d\in\mf{A}$ for some $i$, which is a contradiction.
    \vspace{0.1in}
    
    \noindent \textbf{Case 2:} Suppose $n-1=\alpha$. Since $m<n-1$, let $k>1$ be the smallest integer such that $km-1\geq n-1$. Now, adjusting Lemma~\ref{productRule3}, we get
    \begin{align*}
        (1+x^{(m-1)d})(1+x^{md})^{k-1}(1+x^{\alpha d})^{m} \prod_{i=1}^{\alpha-1}(1+x^{(\alpha+i)d})
        &= (1+x^{(n-1)d}+x^{nd}) + x^{\alpha d}(1+x^d)^{q-2\alpha+km-1} 
        \\
        & +  
        (1+x^{d})\left(\sum_{i=1}^{k-1}(x^{(im-1)d} + x^{(q+im-1)d} )\right)
        \\
        & + (1+x^d+x^{md})x^{(q+(k-1)m-1)d}
    \end{align*}
    lies in $\mc{J}$, giving a contradiction. 
    \vspace{0.1in}
    
    \noindent \textbf{Case 3:} Suppose $n-1<\alpha$. Let $k>1$ be the smallest positive integer such that $kn-1\geq\alpha$. By adjusting Lemma~\ref{productRule3}, we have
    \begin{align*}
        (1+x^{(n-1)d})(1+x^{nd})^{k-1}(1+x^{\alpha d})^{n} \prod_{i=1}^{\alpha-1}(1+x^{(\alpha+i)d})&= (1+x^{(n-1)d}+x^{nd}) + x^{\alpha d}(1+x^d)^{q-2\alpha+kn-1} 
        \\
        & +  
        (1+x^d)\left(\sum_{i=1}^{k-1}(x^{(in-1)d} + x^{(q+in-1)d} )\right)
        \\
        & + (1+x^d+x^{(\alpha+1)d})x^{(q-\alpha+kn-2)d}.
    \end{align*} 
    By case 2, we have $1+x^d+x^{(\alpha+1)d}\in \mc{J}$. Hence, the above product is in $\mc{J}$, a contradiction. 
    \vspace{0.1in}

    Finally, let $n<m$. Let $k$ be the smallest integer such that $kn-1\geq \alpha$. By adjusting Lemma~\ref{productRule3}, we get
    \begin{align*}
        (1+x^{(n-1)d})(1+x^{nd})^{k-1}(1+x^{\alpha d})^{n} \prod_{i=1}^{\alpha-1}(1+x^{(\alpha+i)d})&= (1+x^{(n-1)d}+x^{nd}) + x^{\alpha d}(1+x^d)^{q-2\alpha+kn-1} 
        \\
        & +  
        (1+x^d)\left(\sum_{i=1}^{k-1}(x^{(in-1)d} + x^{(q+in-1)d} )\right)
        \\
        & + (1+x^d+x^{(\alpha+1)d})x^{(q-\alpha+kn-2)d}.
    \end{align*} 
    It suffices to show that $1+x^d+x^{(\alpha+1)d}\in \mc{J}$. If $m-1<\alpha$, then by case 2, the claim holds. If $m-1>\alpha$, then by Lemma~\ref{productRule1}, 
    \[(1+x^{\alpha d})^m\prod_{i=1}^{\alpha-1}(1+x^{(\alpha+i)d}) = (1+x^{\alpha d}+x^{(\alpha+1)d}) + x^{\alpha d}(1+x^d)^{q-2\alpha} + (1+x^d+x^{md})x^{(q-m)d}\in \mc{J}.\]
    Hence, a contradiction and the claim holds. Therefore,
    \[\{1+x^d+x^s\mid s\geq 1\}\subseteq \mc{J}.\]
\end{proof}

\noindent Using the same strategy, we can prove a stronger version, for which we omit the proof. 

\begin{prop}\label{md<ad}
    If there exists a positive integer $m>a$ for which $1+x^{ad}+x^{md}\in \mc{J}$ but $(m-a)d,md\notin \mf{A}$ then the set $\{1+x^{ad}+x^s \mid s\geq 1\}$ is contained in $\mc{J}$.
\end{prop}
\begin{proof}
    If $m<a$ and $ad\in\mf{A}$ with $md\notin\mf{A}$, then we have
    \[(1+x^{(a-m)d})x^{md}+1+x^{ad}=1+x^{ad}+x^{md}\in\mc{J}, \] since $(a-m)d\in\mf{A}$. The proof of the case $m>a$ follows similarly to the proof of Proposition~\ref{1+x^d+x^{md} contained}. 
\end{proof}

\begin{prop}
    Let $|\mf{A}|>1$. For any $ad\in\mf{A}\backslash\{d\}$, we have
    \[\{1+x^{ad}+x^s\mid s\geq 1\}\subseteq \mc{J}\]
    if and only if
    \[\{1+x^d+x^s\mid s\geq 1\}\subseteq\mc{J}.\]
\end{prop}
\begin{proof}
    Suppose $\{1+x^{ad}+x^s\mid s\geq 1\}\subseteq \mc{J}$ for any $ad\in\mf{A}\backslash\{d\}$. By Lemma~\ref{d nmid a}, we have $\{1+x^d+x^s\mid d\nmid s\}\subseteq\mc{J}$. Let $m$ be the smallest integer such that $1+x^d+x^{md}\notin\mc{J}$. Observe that we must have $m\geq 3$. By Lemma~\ref{{m-a,m}_in_J}, $1+x^{(m-1)d}+x^{md}\in\mc{J}$. Consider the following cases:
    
    \vspace{0.1in}
    \noindent \textbf{Case 1:} If $m-1> \alpha$, then by Lemma~\ref{productRule1}, we have
    \[(1+x^{\alpha d})^{m}\prod_{i=1}^{\alpha-1}(1+x^{(\alpha+i)d}) = (1+x^{(m-1)d}+x^{md}) + x^{\alpha d}(1+x^d)^{q-2\alpha} + (1+x^d+x^{(m-1)d})x^{(q-m+1)d}\in\mc{J},\]
    a contradiction as $1+x^d+x^{(m-1)d}\in \mc{J}$ by assumption.

    \vspace{0.1in}
    \noindent \textbf{Case 2:} For the case when $m-1\leq \alpha$. Let $k>1$ be the smallest integer such that $km-1> \alpha$. By adjusting Lemma~\ref{productRule3}, the element
    \begin{align*}
        (1+x^{(m-1)d})(1+x^{md})^{k-1}(1+x^{\alpha d})^{m} \prod_{i=1}^{\alpha-1}(1+x^{(\alpha+i)d}) &= (1+x^{(m-1)d}+x^{md}) + x^{\alpha d}(1+x^d)^{q-2\alpha+km-1} 
        \\
        &+(1+x^{d})\left(\sum_{i=1}^{k-1}(x^{(im-1)d} + x^{(q+(im-1))d} )\right)
        \\
        &+ (1+x^{(\alpha-1)d}+x^{(q-\alpha+km-1)d})x^{\alpha d}
    \end{align*}
    lies in $\mc{J}$ as $1+x^{(\alpha-1)d}+x^{(q-\alpha+km-1)d}\in \mc{J}$ by assumption. This gives a contradiction.

    \vspace{0.1in}    
    \noindent Therefore, we must have $\{1+x^d+x^s\mid s\geq 1\}\subseteq \mc{J}$.

    \vspace{0.1in}
    \indent Conversely, suppose $\{1+x^d+x^s\mid s\geq 1\}\subseteq\mc{J}$ and there exists $ad\in\mf{A}\backslash\{d\}$ and $m$ such that $1+x^{ad}+x^{md}\notin \mc{J}$. Let $m$ be the smallest such integer. By the argument in Proposition~\ref{md<ad}, we must have $m>a$. Also, by Lemma~\ref{{m-a,m}_in_J} we have $1+x^{(m-a)d}+x^{md}\in \mc{J}$.
    Consider the following cases:

    \vspace{0.1in}
    \noindent\textbf{Case 1:} If $m-a\geq\alpha$ then by Lemma~\ref{productRule1}, $(1+x^{\alpha d})^{m}\prod_{i=1}^{\alpha-1}(1+x^{(\alpha+i)d})\in\mc{J}$, a contradiction.

    \vspace{0.1in}
    \noindent \textbf{Case 2:} Suppose $\alpha\leq m<\alpha+a$. Recall that, we have $(m-a)d\notin\mf{A}$. By Lemma~\ref{productRule2}, $(1+x^{(m-a)d})(1+x^{\alpha d})^m\prod_{i=1}^{\alpha-1}(1+x^{(\alpha+i)d})\in\mc{J}$, a contradiction.

    \vspace{0.1in}
    \noindent \textbf{Case 3:} Finally, suppose $m<\alpha$. Recall that $md,(m-a)d\notin\mf{A}$. Let $k>1$ be the smallest integer such that $km-a\geq \alpha$. By adjusting Lemma~\ref{productRule3}, we have
    \begin{align*}
        (1+x^{(m-a)d})(1+x^{md})^{k-1}(1+x^{\alpha d})^{m} \prod_{i=1}^{\alpha-1}(1+x^{(\alpha+i)d}) &= (1+x^{(m-a)d}+x^{md}) + x^{\alpha d}(1+x^d)^{q-2\alpha+km-a} 
        \\
        &+(1+x^{ad})\left(\sum_{i=1}^{k-1}(x^{(im-a)d} + x^{(q+(im-a))d} )\right)
        \\
        &+(1+x^d+x^{(q-\alpha+km-a)d})x^{\alpha d}
    \end{align*}
    lies in $\mc{J}$, a contradiction.
    
    \vspace{0.1in}    
    \noindent Hence by Proposition~\ref{md<ad}, $\{1+x^{ad}+x^s\mid s\geq 1\}$ is contained in $\mc{J}$.
\end{proof}

\begin{lemma}\label{Category II for d=1}
    Let $ad\in\mf{A}$ and let $m>a$ be an integer such that $(m-a)d,md\notin\mf{A}$. If $1+x^{ad}+x^{md}\in\mc{J}$ then $1+x^{(m-a)d}+x^{md}\notin\mc{J}$. Similarly, if $1+x^{(m-a)d}+x^{md}\in\mc{J}$ then $1+x^{ad}+x^{md}\notin\mc{J}$. In particular, we have exactly one of the following inclusions: either
    \[\{1+x^{ad}+x^{s}\mid ad\in\mf{A},\, s\geq 1\}\subseteq \mc{J}\quad\textrm{or}\quad \{1+x^{(m-a)d}+x^{md}\mid ad\in\mf{A},\,m\geq 1\}\subseteq \mc{J}.\]
\end{lemma}
\begin{proof}
    If both $1+x^{ad}+x^{md}\in\mc{J}$ and $1+x^{(m-a)d}+x^{md}\in\mc{J}$, then by Lemma~\ref{productRule1}, Lemma~\ref{productRule2}, and Lemma~\ref{productRule3}, we have a contradiction. Therefore, by the previous propositions, only one of the inclusions is possible.
\end{proof}

\begin{theorem}\label{J=J_A}
    Let $\mc{J}$ be a prime ideal of $\B[x]$ in category (II). If 
    \[\mc{J}_{\mf{A}}=\langle\{1+x^a+x^m\mid a\in\mf{A},m\geq 0\}\rangle\subseteq \mc{J}\]
    then $\mc{J}=\mc{J}_{\mf{A}}$.
\end{theorem}
\begin{proof}
    Suppose 
    \[f(x)\in \mc{J}\backslash \langle\{1+x^a+x^m\mid a\in\mf{A},m\geq 0\}\rangle.\]
    be an irreducible polynomial. Without loss of generality, we can assume $1\in \supp(f)$. Also, $\supp(f)$ cannot contain the term $x^a$ for any $a\in\mf{A}$, since such terms would imply $f\in\<\{1+x^a+x^m\mid a\in\mf{A},\,m\geq 0\}\>$. Let $f(x)=1+x^{m_1d}+\cdots+x^{m_rd}$, with $m_1<\cdots<m_r$. If $m_1\geq \alpha$, then we reach a contradiction by Lemma~\ref{productRule1}. Now, let $\ell\in\{1,\ldots,r\}$ be the largest index such that $m_{\ell}<\alpha$. Consider the following product
    \begin{align*}
        \prod_{j=1}^{\ell}(1+x^{m_jd})^{\alpha}(1+x^{\alpha d})^{m_r^r}\prod_{i=1}^{\alpha-1}(1+x^{(\alpha+i)d})&= 1+ \sum_{A\subseteq\{1,\ldots,\ell\},\,T=\sum_{k\in A}c_km_k<\alpha} x^{Td} + 
        \\
        &x^{\alpha d} + x^{(\alpha+1)d} + \cdots + x^{(q-\alpha+\alpha\sum_{i=1}^{\ell} m_i)d}\\
        &+x^{qd}\left(1+ \prod_{j=1}^{\ell}(1+x^{m_jd})^{\alpha}\right),
    \end{align*}
    where $q=m_r^r\alpha+\sum_{i=1}^{\alpha-1}(\alpha+i)$. Note that the above product can also be written as
    \begin{align*}
        &f(x)\left(1+\sum_{A\subseteq\{1,\ldots,\ell\},\,T=\sum_{k\in A}c_km_k<\alpha} x^{Td}\right)+x^{\alpha d}(1+x^d)^{q-2\alpha+\alpha\sum_{i=1}^{\ell} m_i} 
        \\
        & + \left(1+x^d+x^{(q-\alpha)d}\prod_{j=1}^{\ell}(1+x^{m_jd})^{\alpha} \right)x^{\alpha d},
    \end{align*}
    which is an element of $\mc{J}$ as $\{1+x^d+x^s\mid s\geq 1\}\subseteq \mc{J}$. But then this gives us a contradiction, as $\mc{J}$ is a prime ideal. Hence, we have $\mc{J }=\mc{J}_{\mf{A}}.$
\end{proof}

For the remainder of the section, we say $f(x)=1+x^{m_1d}+\cdots+x^{m_rd}$ satisfy $\star$ if $m_id\in\mf{A}$ for some $i$ but $\{m_1d,\ldots,m_rd\}\not\subset \mf{A}$ and there exists $t$ such that $|m_t-m_j|d\notin\mf{A}$ for all $j=1,\ldots,r$.  

\begin{theorem}
    Let $\mc{J}$ be a prime ideal of $\B[x]$ in category (II) such that 
    \[\langle\{1+x^{m}+x^{m+a}:a\in\mf{A}, m\geq 0\}\rangle\subseteq \mc{J}.\]
    Then 
    \[\mc{J}= \langle\{1+x^{m}+x^{m+a}:a\in\mf{A}, m\geq 0\},\mc{Q}\rangle\]
    for some $\mc{Q}\subseteq\{f\mid f\textrm{ is irreducible and satisfy $\star$}\}$.
\end{theorem}
\begin{proof}
    Let $f=1+x^{m_1d}+\cdots+x^{m_rd}\in \mc{J}$ be irreducible. We claim that 
    \[f\notin \langle \{1+x^{m}+x^{m+a}:a\in\mf{A}, m\geq 0\}\rangle\]
    if and only if $f$ satisfy $\star$. Let $f\notin \langle \{1+x^{m}+x^{m+a}:a\in\mf{A}, m\geq 0\}\rangle$. Note that we must have $\{m_1d,\ldots,m_rd\}\not\subset \mf{A}$. Suppose $m_id\notin\mf{A}$ for all $i$. Then by the same argument as in Theorem~\ref{J=J_A}, we have a contradiction. Also, if for all $t$, there exists $j_t$ such that $|m_t-m_{j_t}|d\in \mf{A}$ then 
    \[f=\sum_{t=1}^r(1+x^{m_td}+x^{m_{j_td}})\in \langle \{1+x^{m}+x^{m+a}:a\in\mf{A}, m\geq 0\}\rangle.\]
    The opposite direction is straightforward. Hence, $\langle\{1+x^{m}+x^{m+a}:a\in\mf{A}, m\geq 0\},f\rangle\subseteq \mc{J}$.
\end{proof}

\subsection{The case $d\notin\mf{A}$}   

\begin{theorem}\label{d not in A and A inf}
    Let $\mc{J}$ be a prime ideal of $\B[x]$ in category (II). Then 
    \[\mc{J}=\mc{J}_{\mf{A}}.\]
\end{theorem}
\begin{proof}
    Since $d\notin\mf{A}$, we have $\N-\mf{A}=\<d\>$. For any $m> a$, if $m,m-a\notin \mf{A}$, then $d|a$, a contradiction. So, we must have either $m\in\mf{A}$ or $m-a\in\mf{A}$. So, by Lemma~\ref{m or m-a in A implies containtment}, we have 
     \[\{1+x^{a}+x^{s}\mid a\in\mf{A},\, s\geq 1\}\subseteq \mc{J}\quad\textrm{and}\quad \{1+x^{m}+x^{m+a}\mid a\in\mf{A},\,m\geq 1\}\subseteq \mc{J}.\]
    In particular, $\mc{J}_{\mf{A}}\subseteq\mc{J}$. Also, it is easy to observe that 
    \[\mc{J}_{\mf{A}}\supset \{1+x^{m}+x^{m+a}\mid a\in\mf{A},\,m\geq 1\}.\]
    Now, let $f=1+x^{m_1}+\cdots+x^{m_r}\in\mc{J}\backslash\mc{J}_{\mf{A}}$. We must have $m_i\notin\mf{A}$ for all $i$, which implies $d|m_i$ for all $i$. So, there exists large $N$ and $M$ such that $f\cdot(1+x^d)^N=(1+x^d)^M\in \mc{J}$. But then this is a contradiction as $d\notin\mf{A}$. Hence, $\mc{J}=\mc{J}_{\mf{A}}$.
\end{proof} 

This completes the classification of prime ideals in the polynomial semiring $\B[x]$. Combining our results, we get Theorem \ref{D}.

\begin{corollary}
    If $d=1$, and let $\mc{J}$ be a prime ideal of $\B[x]$ in category (II). Then either $\mc{J}=\mc{J}_{\mf{A}}$ or
    \[\mc{J}=\langle \{1+x^m+x^{m+a} \mid a\in\mf{A},\,m\geq 0\},\mc{Q}\rangle\]
    for some $\mc{Q}\subseteq\{f=1+\sum_{i=1}^rx^{m_i}\mid \textrm{$f$ is irreducible and satisfy $\star$}\}.$
\end{corollary}

\section{Prime ideals of $\B[x_1,\cdots,x_n]$} \label{Future direction}

Let $\N=\{0,1,2,\ldots\}$. In this section, we provide partial results for prime ideals in $\B[x_1,\cdots,x_n]$. 
Throughout the section, we use the notation $\vec{a}$ to denote elements of $\N^n$ and by $X^{\vec{a}}$ we mean $x_1^{a_1}x_2^{a_2}\cdots x_n^{a_n}$ when $\vec{a}=(a_1,\cdots,a_n)$. A subset $\mf{A}\subseteq\N^n$ is said to be a \emph{prime subset} if for any $\vec{a},\vec{b}\in\N^n$ such that $\vec{a}+\vec{b}\in \mf{A}$, we have $\vec{a}\in \mf{A}$ or $\vec{b}\in \mf{A}$. 

We now generalize the methods due to F. Alarcón and D. Anderson \cite{AA94} to the multivariable case. If $\mf{A}=\emptyset$, then we use the convention $\mc{I}_{\emptyset}=\langle x_1,\cdots,x_n\rangle$.

\begin{theorem}
    Let $\mf{A}$ be a subset in $\N^n$. Consider the polynomial semiring $\B[x_1,x_2,\cdots,x_n]$. To the set $\mf{A}$ we associate the ideal $\mc{I}_\mf{A}$ of $\B[x_1,x_2,\cdots,x_n]$ defined as
    \[\mc{I}_\mf{A}:=\langle x_1,x_2,\cdots,x_n, \{1+X^{\vec{a}}\mid \vec{a}=(a_1,\cdots,a_n)\in \mf{A}\} \rangle.\]
    Then the ideal $\mc{I}_\mf{A}$ can also be described as
    \[\mc{I}_\mf{A}=\left\{ f = \sum_{\substack{\vec{a}\in\N^n, \\ \textrm{finite}}} X^{\vec{a}} \mid \textrm{ if $1\in\supp(f)$}  \textrm{ then } \exists\, \vec{a}\in \mf{A} \textrm{ such that } X^{\vec{a}}\in\supp(f)  \right\}.\]
    Then 
    \begin{enumerate}
        \item The ideal $\mc{I}_\mf{A}$ is prime if and only if $\mf{A}$ is a prime subset of $\N^n$.
        \item If $\mc{P}$ is a prime ideal containing $\langle x_1,x_2,\cdots,x_n\rangle$, then $\mc{P}=\mc{I}_\mf{A}$ where 
        \[\mf{A}=\{\vec{a}\in\N^n\mid 1+X^{\vec{a}}\in \mc{P}\}\]
        is a prime subset of $\N^n$. 
    \end{enumerate}
\end{theorem}
\begin{proof}
    
    To prove (1), let $\mf{A}$ be a subset of $\mathbb{N}^n$ and consider the ideal $\mc{I}_{\mf{A}}$ as defined above. Suppose $\mc{I}_\mf{A}$ is a prime ideal. Let $\vec{a}$ and $\vec{b}$ be elements in $\N^n$ such that $\vec{a}+\vec{b}\in \mf{A}$. Then
        \[(1+X^{\vec{a}})(1+X^{\vec{b}})= 1+X^{\vec{a}}+X^{\vec{b}}+X^{\vec{a}+\vec{b}} \in \mc{I}_\mf{A}.\]
        Since $\mc{I}_\mf{A}$ is prime, we have $(1+X^{\vec{a}})\in \mc{I}_\mf{A}$ or $(1+X^{\vec{b}})\in \mc{I}_\mf{A}$, which implies $\vec{a}\in \mf{A}$ or $\vec{b}\in \mf{A}$.
        
        Conversely, let $\mf{A}$ be a prime subset in $\N^n$ and consider the ideal $\mc{I}_\mf{A}$. Let $f=\sum_i X^{\vec{a}_i}$ and $g=\sum_i X^{\vec{b}_j}$ be arbitrary elements of $\B[x_1,\cdots,x_n]$ such that 
        \[fg=\sum_k X^{\vec{c}_k}\in I_\mf{A}.\]
        If $1\notin \supp(fg)$, then we have $1\notin \supp (f)$ or $1\notin\supp(g)$, in which case $f\in\mc{I}_{\mf{A}}$ or $g\in\mc{I}_{\mf{A}}$. If $1\in\supp(fg)$, then there exists $k$ such that $\vec{c}_k\in\mf{A}$, where $\vec{c}_k$ is one of the exponents of $fg$. But then there exists $i$ and $j$ for which $\vec{a}_i+\vec{b}_j=\vec{c}_k$. Now, since $\mf{A}$ is a prime subset, we must have $\vec{a}_i\in\mf{A}$ or $\vec{b}_j\in\mf{A}$. As $1\in \supp(f)\cap\supp(g)$, we have $f\in\mc{I}_{\mf{A}}$ or $g\in\mc{I}_{\mf{A}}$. Therefore, $\mc{I}_{\mf{A}}$ is a prime ideal.
        
        To prove (2), let $\mc{P}$ be a prime ideal containing $\langle x_1,\cdots,x_n\rangle$ and consider the set 
        \[\mf{A}=\{\vec{a}\in\N^n\mid 1+X^{\vec{a}}\in \mc{P}\}.\]
        If $\mf{A}=\emptyset,$ we have nothing to prove. Suppose $\mf{A}\neq \emptyset$. Then clearly $\mc{I}_\mf{A}$ is contained in $\mc{P}$. Let $f=\sum_iX^{\vec{a}_i}\in\mc{P}\backslash \mc{I}_{\mf{A}}$. Observe that we must have $\vec{a}_i\notin\mf{A}$ for all $i$. Then the product
        \[\prod_{i}(1+X^{\vec{a}_i})=f+h\in\mc{P}\]
        where $h\in\< x_1,\cdots,x_n\>$. Since $\mc{P}$ is prime, we have $1+X^{\vec{a}_i}\in\mc{P}$ for some $i$. In particular, $\vec{a}_i\in\mf{A}$ for some $i$, a contradiction. Therefore, $\mc{P}=\mc{I}_{\mf{A}}$.
  
\end{proof}

Note that, in the multivariable case, the set
\begin{align*}
    \mc{J}_{\mf{A}}:&=\mc{I}_{\mf{A}}-\{X^{\vec{a}}f\mid \vec{a}\in\N^n,\,f\notin\mc{I}_{\mf{A}}\}.
    \\
    &= \{X^{\vec{a}}f\mid \vec{a}\in\N^n,\,f\in\mc{I}_{\mf{A}}\textrm{ with $1\in\supp(f)$}\}\cup\{0\}.
\end{align*}
is not even an ideal. Furthermore, we do not know whether prime subsets of $\N^n$ parametrize all the prime ideals of $\B[x_1,\cdots,x_n]$. To illustrate the complexity of classifying primes in the multivariable case, we give two distinct prime ideals in $\B[x,y]$ for which $\mf{A}=\emptyset$.

\begin{example}
    The ideals 
    \[\mc{P}_1=\<  \{x^a+y^b+x^cy^d \mid a,b\geq 1\textrm{ and }c,d\geq 0\}\>.\]
    and
    \[\mc{P}_2=\<  \{x^a+y^b\,:\,a,b\geq 1\},\{1+x^{p-a}y^q+x^py^{q-b}\,:\,p,q\geq 1\textrm{ and }p-a,q-b\geq 0\}\>.\]
    are prime. We can prove that they are prime ideals by using the equality 
    \begin{align*}
        (x^a+y^b+x^py^q)(x^ay^b+x^{p+a}y^q+x^py^{q+b})&=(x^a+y^b)(x^a+x^py^q)(y^b+x^py^q) 
    \end{align*}
    and
    \begin{align*}
        (1+x^ay^b+f(x,y))(x^ay^b+f(x,y)+x^ay^bf(x,y))&=(1+x^ay^b)(1+f(x,y))(x^ay^b+f(x,y))
    \end{align*}
    in $\B[x,y]$, for all $a,b\geq 1$ and $p,q\geq 0$. 
\end{example}

\begin{remark}
    Using the results of \cite{JM25} we can produce a large number of examples of prime ideals, called 'closed' prime ideals. In the case of $\B[x_1, \dots, x_n]$ we can list all closed prime ideals. We give an example in the case $\B[x,y]$: all close prime ideals correspond to a (non-minimal) prime congruence, which in turn correspond to a matrix $[a, b]$, where $a, b\in \mathbb{Q}$, which is in turn equivalent to a monomial pre-order on monomials in two variables. Each closed prime ideal is generated by polynomials of the form $\sum_i x^{n_i}y^{m_i}$, where after reordering $n_1a+ m_1b = n_2a+ m_2b \geq n_ja+ m_jb,$ for $j > 2$. Since not all primes are closed primes, e.g. in $\B[x]$ there is a single closed prime ideal, this is an incomplete classification as well. 
\end{remark}

\bibliographystyle{alpha}

\end{document}